\documentclass[12pt]{arxivclassfile}

\usepackage{mathrsfs,mathtools}

\usepackage{amsmath,amssymb,amsthm}

\usepackage{graphicx}

\newtheorem{thm}{Theorem}[section]
\newtheorem{lem}[thm]{Lemma}

\theoremstyle{definition}

\theoremstyle{remark}

\numberwithin{equation}{section}

\def\R{{\mathbb R}}

\def\<{\langle}
\def\>{\rangle}

\usepackage{esint}

\def\R{{\mathbb R}}

\def\H{{\mathbb H}}
\def\V{{\mathbb V}}
\def\X{{\mathbb X}}
\def\C{{\mathbb C}}
\def\L{{\mathbb L}}

\def\cgw{\rightharpoonup}

\def\d{{\rm d}}

\def\uu{{\mathbf u}}
\def\vv{{\mathbf v}}
\def\ff{{\mathbf f}}
\def\ww{{\mathbf w}}
\def\vp{{\varepsilon}}

\def\hh{{\mathbf h}}

\def\per{{\rm per}}

\def\be#1{\begin{equation}\label{#1}}
\def\ee{\end{equation}}

\def\email#1{\ead{#1}}

\usepackage{graphicx}

\def\be#1{\begin{equation}\label{#1}}
\def\ee{\end{equation}}

\def\og{\leavevmode\raise.3ex\hbox{$\scriptscriptstyle\langle\!\langle$~}}
\def\fg{\leavevmode\raise.3ex\hbox{~$\!\scriptscriptstyle\,\rangle\!\rangle$}}

\begin{document}

%-------------------------------------------------------------------------
% editorial commands: to be inserted by the editorial office
%
%\firstpage{1} \volume{228} \Copyrightyear{2004} \DOI{003-0001}
%\seriesextraline{This is the Concrete Title of this Book\br H.E. R and S.T.C. W, Eds.}
%\firstpage{1}

\title{Limits of the Stokes and Navier--Stokes equations in a punctured periodic domain}

%----------Author 1
\author[Z]{Michel Chipot}
\address[Z]{Institut f\"{u}r Mathematik, Angewandte Mathematik, Winterthurerstrasse 190, CH-8057 Z\"{u}rich.
Switzerland}
\email{m.m.chipot@math.uzh.ch}

%\thanks{This work was completed with the support of our

\author[JD]{J\'er\^ome Droniou}
\address[JD]{School of Mathematical Sciences\\
Monash University\\
Clayton, Victoria 3800\\
Australia.}
\email{jerome.droniou@monash.edu}

\author[GP]{Gabriela Planas}
\address[GP]{Departamento de Matem\'atica,
Instituto de Matem\'atica, Estat\'\i stica e Computa\c{c}\~ao Cient\'\i fica,
Universidade Estadual de Campinas,
Rua Sergio Buarque de Holanda 651,
13083-859 Campinas, SP.
Brasil}
\email{gplanas@ime.unicamp.br}

\author[W]{James C. Robinson}
\address[W]{%
Mathematics Institute,
University of Warwick,
Coventry, CV4 7AL.
U.K.}
\email{j.c.robinson@warwick.ac.uk}

\author[Z]{Wei Xue}
%\address{%
%Institut f\"{u}r Mathematik,
%Angewandte Mathematik,
%Winterthurerstrasse 190,
%CH-8057 Z\"{u}rich.
%Switzerland}
\email{wei.xue@math.uzh.ch}

%%----------classification, keywords, date
%\subjclass{Primary 99Z99; Secondary 00A00}

%\keywords{Singular limits, Stokes equations, Navier--Stokes equations.}

%\date{January 1, 2004}

\begin{abstract}
We treat three problems on a two-dimensional `punctured periodic domain': we take $\Omega_r=(-L,L)^2\setminus D_r$, where $D_r=B(0,r)$ is the disc of radius $r$ centred at the origin. We impose periodic boundary conditions on the boundary of $\Omega=(-L,L)^2$, and Dirichlet boundary conditions on the circumference of the disc. In this setting we consider the Poisson equation, the Stokes equations, and the time-dependent Navier--Stokes equations, all with a fixed forcing function $f$,and examine the behaviour of solutions as $r\to0$. In all three cases we show convergence of the solutions to those of the limiting problem, i.e.\ the problem posed on all of $\Omega$ with periodic boundary conditions.
\end{abstract}

%%% ----------------------------------------------------------------------
\maketitle
%%% ----------------------------------------------------------------------
%\tableofcontents

\section{Introduction}

 The study of fluid  flow around an obstacle is a challenging and interesting problem in fluid mechanics, and has been the subject of much experimental and numerical investigation (see, among others,  \cite{Ariel,Elde,Homen,Hudson,Pierce,Smith,Turfus,Zannetti}).

The mathematical analysis of the influence of an obstacle on the behaviour of the flow when the size of the obstacle is small when compared to that of the reference spatial scale
has recently received increased attention. The case of a single obstacle in a two-dimensional ideal flow was analysed by Iftimie, Lopes Filho, \& Nussenzveig Lopes \cite{ILL2003}; then  Iftimie et al. \cite{ILL2006} and Iftimie \& Kelliher \cite{IK2009}  considered the viscous case, Lopes Filho \cite{Lopes} treated bounded domains with several holes, Lacave \cite{Lacave09,Lacave2,Lacave3} considered
obstacles that shrink to a curve. For problems in exterior domains (i.e.\ extending to infinity) the flow is usually assumed to vanish at infinity, although the case of flows constant at infinity has been considered by Lopes Filho, Nguyen, \& Nussenzveig Lopes \cite{Lopes13}. A related `small body' problem was considered by Robinson \cite{Robinson}, who treated a simplified model of combustion in which physical particles were replaced by diffuse but compact regions of influence in the flow.
Very recently, Lu \cite{Lu16} treated the Dirichlet problem in the three-dimensional unit ball with a shrinking hole. Uniform estimates, as the size of the hole goes to zero,  in $W^{1,p}$  for  any $ 3/2 < p < 3$  and  counterexamples that the uniform $W^{1,p}$-estimates do not hold when $ 1 < p < 3/2 $ or $ 3 < p < + \infty $ are provided. These estimates were extended by the same author \cite{Lupreprint} to the Stokes problem in a $n$-dimensional bounded domain, showing  uniform estimates  for  any $ n' < p < n$ and  counterexamples for  $ 1 < p < n' $ or $ n < p < + \infty $. Notice that last two papers do not consider the two-dimensional case for $ p=2 $.

Here we are interested in the vanishing obstacle problem in a two-dimensional periodic domain with a particularly simple geometry. More precisely, we are concerned with periodic flows on the punctured domain \[\Omega_r=(-L,L)^2\setminus D_r, \ L > 0 ,\]
where   $D_r=B(0,r)$ is the disc of radius $r$ centred at the origin, and we study the behaviour of the solutions of various models when the radius $ r $ of the disc tends to zero. Throughout the paper we refer to the excised disc $D_r$ as the `obstacle' in keeping with the ultimate application to problems of fluid flow.

Our primary motivation for this geometry  was the moving `tracer particle' problem considered in two dimensions by Dashti \& Robinson \cite{DR} and in three dimensions by Silvestre \& Takahashi \cite{Silvestre}: given a solid disc/sphere of radius $r$ moving in the fluid, does the motion of the particle follow that of the fluid in the limit $r\to0$? Our aim was to include rotation of the tracer in the 2D case, which was excluded in \cite{DR}. However, in the course of the analysis that follows we observed the failure of certain uniform elliptic regularity estimates that are required in both these papers (see Section 2.1); while the two-dimensional case has now been resolved by Lacave \& Takahashi \cite{Lacave17} for small initial data (using maximal regularity estimates for the Stokes equation) the three-dimensional case remains open.
(We choose a particularly simple geometry and a somewhat simpler problem in which these uniform estimates fail, but there is no reason to believe that this has any significant effect of the nature of this phenomenon.)

In order to clarify the setting and provide some background to these uniform elliptic estimates, as well as allowing us to outline the main ideas that will then be applied in the more complicated Stokes and time-dependent Navier--Stokes problems (which have the added component of incompressibility) we first consider the Poisson equation as a model problem. Thus our initial aim (in Section 2) will be to determine the asymptotic behaviour of the solution of the following problem when $ r \rightarrow 0$:
\begin{equation} \label{laplace}
-\Delta u_r = f \ \;\text{in } \Omega_r,\qquad
u_r \text{ periodic},\qquad
u_r = 0 \text{ on } \partial D_r.
\end{equation}
While this problem has a solution for any $f\in L^2(\Omega_r)$, the limiting problem,
$$
-\Delta u=f\ \;\text{in }\Omega,\qquad u \text{ periodic},
$$
only has a solution when
\begin{equation}\label{intf0}
  \int_\Omega f=0.
\end{equation}
We will show that when (\ref{intf0}) holds then the solutions of (\ref{laplace}) are uniformly bounded in $r$ in the sense that
$$
\int_{\Omega_r}|\nabla u_r|^2+\int_{\Omega_r}\left|u_r-\fint_\Omega u_r\right|^2
$$
is uniformly bounded, where $\fint_\Omega u=|\Omega|^{-1}\int_\Omega u$ denotes the average of $u$ over $\Omega$ (note that this is the whole domain and not just $\Omega_r$). This is enough to show that
$$
u_r-\fint_\Omega u_r\to u
$$
in $H^1(\Omega)$ and that $u$ satisfies the limiting equation. If (\ref{intf0}) does not hold then the limiting problem has no solution, and in this case it follows that $\|u_r\|_{H^1}$ is unbounded as $r\to0$.

We remark here, and will return to this later, that we have been unable to obtain a uniform bound on $\fint_\Omega u_r$, since the constant in the Poincar\'e inequality available on $\Omega_r$ degrades as $r\to0$ (see Lemma \ref{badpoinc}).

In Section 3 we obtain similar results for the Stokes problem
\begin{flalign*}
\begin{cases}
-\Delta \uu_r + \nabla p_r=\ff\ \;\text{in } \Omega_r,\\
\text{div } \uu_r = 0 , \\
\uu_r \text{ periodic}, \\
\uu_r = 0 \text{ on } \partial D_r.
\end{cases}
\end{flalign*}
The main change from the case of the pure Laplacian is that we now have to deal with divergence-free vector-valued functions. The key technical result that allows us to do this is a method for approximating divergence-free periodic functions defined on the whole of $\Omega$ by a sequence of divergence-free functions that satisfy the zero boundary condition on $D_r$ (Lemma \ref{approx2}). Once again, we require that $\int_\Omega\ff=0$. As before, we can find uniform estimates sufficient to show that $\uu_r-\fint_\Omega\uu_r$ converges to a solution of the limiting problem, but we are unable to bound the average of $\uu_r$ over $\Omega$.

It would seem that the next natural step would be to consider the stationary Navier--Stokes equations in $\Omega_r$,
\begin{equation}\label{statNSE}
-\Delta \uu_r+(\uu_r\cdot\nabla)\uu_r+\nabla p_r=\ff,\qquad\nabla\cdot\uu_r=0.
\end{equation}
However, while in the linear problems considered so far bounds on $\uu_r-\fint_\Omega\uu_r$ were sufficient to pass to the limit, this is not the case here. Informally, if we set $\<\uu_r\>=\fint_\Omega\uu_r$ and consider the equation for $\tilde\uu_r=\uu_r-\<\uu_r\>$ then we obtain
$$
-\Delta\tilde\uu_r+(\tilde\uu_r\cdot\nabla)\tilde\uu_r +(\<\uu_r\>\cdot\nabla)\tilde\uu_r+\nabla p_r=\ff,
$$
which contains the additional term $(\<\uu_r\>\cdot\nabla)\tilde\uu_r$. A uniform bound on $\<\uu_r\>$ would enable us to pass to the limit in this term, but we do not currently have such a bound.

An additional factor that makes this problem different in character from the others we consider here is that there is no known general uniqueness result for solutions of (\ref{statNSE}), even on the entire periodic domain. As such, it is perhaps more natural to consider a perturbation problem (given a solution of the equation on $\Omega$, investigate the existence of nearby solutions for $r$ small) than as a limiting problem; or to treat a restricted setting in which uniqueness results are available (when $\ff$ is small in an appropriate sense). For more discussion of this stationary problem we refer to the classical work of Ladyzhenskaya \cite{Ladyz} and Temam \cite{Temam1,Temam2}.

%
%The main technical difficulty is the lack of a uniform $L^2$ bound for the weak solution.
%It is worth to notice that there is no restriction on the average of the
%solution for the expected limiting problem, i.e., for the stationary Navier--Stokes equations with periodic boundary conditions in $\Omega$. Moreover, the constant in the Poincar\'e inequality degrades with $ r \to 0$ (see Lemma \ref{badpoinc}), so we can not employ it to obtain the $L^2$ bound from the estimate on the gradient of the solution.

We therefore instead turn in Section 4 to the time-dependent Navier--Stokes problem, which turns
out to be more straightforward and for which we do not require the use of the Poincar\'e inequality, since a bound on the $\L^2$ norm follows immediately from the energy inequality. In this case we obtain convergence of $\uu_r$ to the solution $\uu$ of the periodic Navier--Stokes equations,
$$
\partial_t\uu-\Delta\uu+(\uu\cdot\nabla)\uu+\nabla p=\ff,\qquad \nabla\cdot\uu=0,
$$
where the convergence is strong in $L^2(0,T;\L^2(\Omega))$ and weak in $L^2(0,T;\H^1(\Omega))$. We note that this falls short of $\L^\infty$ convergence of the velocity field; this is unsurprising since uniform convergence coupled with the fact that $\uu_r=0$ on $\partial D_r$ would imply that the limiting flow was stationary at the origin.

\section{Poisson equation}

In this section we discuss the asymptotic behaviour of weak solutions for the Poisson problem
$$
\begin{cases}
-\Delta u_r = f \ \;\text{in } \Omega_r,\\
u_r \text{ periodic}, \\
u_r = 0 \text{ on } \partial D_r.
\end{cases}
$$

Let us introduce some notation. Set $\Omega_0=(-L,L)^2=\Omega$ and $\Omega_r=(-L,L)^2\setminus D_r$, where $D_r=B(0,r)$ is the disc of radius $r$. We use the subscript `per' on a space $X$ to denote the restriction to $\Omega$ (or to $\Omega_r$) of a function that is $2L$-periodic on $\R^2$ in both directions and is in $X_{\rm loc}(\R^2)$. In this way we  define the function spaces
$H^1_\per(\Omega)$ and, for $ r > 0$,
$$
H^1_{\per}(\Omega_r)=\text{ the closure of }C_{\per}^1(\overline \Omega_r) \text{ in } H^1 (\Omega_r)
$$
and
$$
V_{0,r}=\{v\in H^1_\per(\Omega_r):\ v=0\ \mbox{on}\ \partial D_r\}.
$$
Note that any function in $V_{0,r}$ can be extended by zero inside $D_r$ to give a function in $H^1_\per(\Omega)$; this observation is fundamental to our analysis.

The vanishing obstacle problem for the Poisson equation
\begin{equation} \label{laplacebis}
-\Delta u_r = f \ \;\text{in } \Omega_r,\qquad
u_r \in V_{0,r},
\end{equation}
consists in determining the asymptotic behaviour of the solution $ u_r $ when $r$ tends to $0$.

The precise statement of our first convergence result is as follows.

\begin{thm}\label{mainlaplace}
Let $f\in L^2(\Omega)$.  For every $r>0$ there exists a unique solution $u_r\in V_{0,r}$ of the problem
  \begin{equation}\label{Lapr}
  \int_{\Omega_r}\nabla u_r\cdot\nabla v=\int_{\Omega_r}fv\qquad\mbox{for all}\ v\in V_{0,r}.
  \end{equation}
  Moreover
\begin{itemize}
\item[a)]   if $\int_\Omega f=0$ then as $r\to0$
  $$
  u_r-\frac{1}{|\Omega|}\int_\Omega u_r\to u_0\qquad\mbox{and}\qquad \nabla u_r\to\nabla u_0,
$$
where the limits are taken in $L^2(\Omega)$ and $u_0 \in H^1_\per(\Omega)$ is the unique solution of the problem
  \begin{equation}\label{Lap0}
  \int_\Omega\nabla u_0\cdot\nabla v=\int_\Omega fv\qquad\mbox{for all}\ v\in H^1_\per(\Omega)
  \end{equation}
  that satisfies $\int_\Omega u_0=0$.

  \item[b)] If $\int_\Omega f\neq0$ then $\|\nabla u_r\|_{L^2}$ is unbounded as $r\to0$. %More precisely, $\|\nabla u_r\|_{L^2}$ increases as $r$ decreases.
  \end{itemize}
\end{thm}

A few comments are in order.

Note that one can use $v=1$ as a test function in \eqref{Lap0}, from which it follows immediately that there can be no solution of the limiting problem unless
  $$
  \int_\Omega f=0.
  $$
% As we mentioned  in the Introduction,  we will see as a consequence of this that if $\int_\Omega f\neq0$ then there can be no uniform bound on $\|u_r\|_{H^1}$.

Observe that we do not have convergence of $u_r $ itself in $L^2(\Omega)$. The main reason for this is that the constant in the Poincar\'e inequality for the punctured domain $ \Omega_r$ degrades as $ r \to 0$. We first recall the classical Poincar\'e inequality: there exists a constant $C>0$ such that for any $v\in H^1_\per(\Omega)$
\begin{equation}\label{perpoinc}
\left\|v-\fint v\right\|_{L^2(\Omega)}\le C\|\nabla v\|_{L^2(\Omega)},
\end{equation}
where
$$
\fint v=\frac{1}{|\Omega|}\int_\Omega v.
$$
Notice that  inequality \eqref{perpoinc} is still valid for functions in $v\in V_{0,r}$, and in particular the constant does not depend on $r$. However, without subtraction of the average we have only the following estimate.

\begin{lem}\label{badpoinc}
Let $r<(2-\sqrt2)L$. Then for all $v\in V_{0,r}$
  $$
  \|v\|_{L^2(\Omega_r)}\le c(-\log r)\|\nabla v\|_{L^2(\Omega_r)}.
  $$
\end{lem}

\begin{proof} We assume that $v\in C^1_{\per}(\overline\Omega_r)$ with $v=0$ on $\partial D_r$, with the result for $v\in V_{0,r}$ obtained by a density argument. We extend $v$ periodically outside $\Omega_r$, the assumption that $r<(2-\sqrt 2)L$ meaning that any $x$ with $|x|\le\sqrt2L$ in the extended domain does not lie within one of the additional `holes', see Figure 1.

\begin{figure}[hbt]
\centering
 \includegraphics[width=0.5\textwidth]{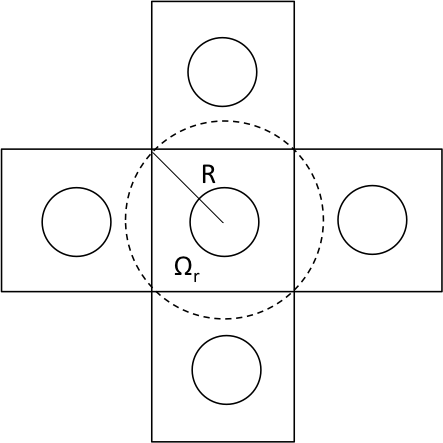}
 \caption{Periodic extension of the domain $\Omega_r$ used in the proof of Lemma 2.2}
\end{figure}

At $x=\rho\hat x$ (where $\hat x=x/|x|$), we can write
  \begin{align*}
  v(x)&=v(\rho\hat x)-v(r\hat x)=\int_r^\rho \frac{\d}{\d s}v(s\hat x)\,\d s\\
  &\le\int_r^\rho|\nabla v(s\hat x)|\,\d s.
  \end{align*}
  Then, since $B(0,\sqrt2 L)\supset\Omega_r$, setting $R=\sqrt2 L$ we have
  \begin{align*}
  \int_{\Omega_r}|v(x)|^2&\le\int_0^{2\pi}\int_r^R \rho|v(\rho\hat x)|^2\,\d \rho\,\d\theta\\
  &\le\int_0^{2\pi}\int_r^R \rho\left(\int_r^\rho|\nabla v(s\hat x)|\,\d s\right)^2\,\d\rho\,\d\theta\\
  &\le\int_0^{2\pi}\int_r^R \rho\left(\int_r^\rho s^{-1}\,\d s\right)\left(\int_r^\rho s|\nabla v(s\hat x)|^2\,\d s\right)\,\d\rho\,\d\theta\\
  &\le\int_0^{2\pi}\int_r^R \rho\log(\rho/r)\left(\int_r^\rho s|\nabla v(s\hat x)|^2\,\d s\right)\,\d\rho\,\d\theta\\
  &\le\left(\int_r^R\rho\log(\rho/r)\,\d\rho\right)\left(\int_{B(0,R)}|\nabla v|^2\,\d x\right),\\
  &\le c(-\log r)\|\nabla v\|_{L^2(\Omega_r)}^2,
  \end{align*}
  using the fact that $\int_{B(0,R)}|\nabla v|^2\le 2\int_{\Omega_r}|\nabla v|^2$ since we have extended $v$ periodically outside $\Omega_r$.
\end{proof}

We note that the fact that the constant in Lemma \ref{badpoinc} is not independent of $r$ is not merely an artefact of our method of proof: while it may be possible to improve the dependence on $r$, one cannot remove it. Indeed, consider the family of functions $u_r$ defined on $\Omega_r$ by
$$
u_r(x)=\log(1+\log(\rho/r))
$$
where $\rho$ is distance of $x$ from the origin. This defines a function in $V_{0,r}$, since its values on the boundary of $\Omega$ agree on opposite faces.

Now, certainly
\begin{align*}
\|u_r\|_{L^2(\Omega_r)}^2\ge\int_{r\le|x|\le L}|u_r(x)|^2\,\d x&=\int_r^L \rho(\log(1+\log(\rho/r)))^2\,\d\rho\\
&=r^2\int_1^{L/r}s(\log(1+\log s))^2\,\d s\\
&\ge r^2\int_{L/2r}^{L/r}s(\log(1+\log s))^2\,\d s\\
&\ge r^2(L/2r)^2\log(1+\log(L/2r))^2\\
&= \frac{L^2}{4}\log(1+\log(L/2r))^2,
\end{align*}
which is unbounded as $r\to0$. However,
$$
\partial_\rho u_r=\frac{1}{1+\log(\rho/r)}\frac{1}{\rho}
$$
and so
\begin{align*}
\|\nabla u_r\|_{L^2(\Omega_r)}^{2}\le\int_{r\le|x|\le\sqrt2 L}|\partial_\rho u_r|^2\,\d x
&=\int_r^{\sqrt2 L}\frac{1}{(1+\log(\rho/r))^2}\frac{1}{\rho}\,\d\rho\\
&\le\int_1^{\infty}\frac{1}{s(1+\log s)^2}\,\d s<\infty.
\end{align*}

We now state  a preliminary lemma on approximation of functions in $H^1_\per(\Omega)$ by functions in $V_{0,r}$ which will be used to pass to the limit.

\begin{lem}\label{approx}
  Given $v\in H^1_\per(\Omega) $ there exists a sequence $v_\vp\in V_{0,\vp}$ such that
  $$
  v_\vp\to v\quad\mbox{in}\quad H^1 (\Omega)\qquad\mbox{as}\qquad\vp\to0.
  $$
\end{lem}

\begin{proof}
(The proof consists essentially of showing that $\{0\}$
has zero $2$-capacity in $\R^2$, see Heinonen, Kilpel\"ainen, \& Martio \cite{HKM}.)

Without loss of generality, we can assume that $ 0 < \vp < 1$.
Let
  $$
  \phi_\vp(x)=\min(1,(-\ln\vp)^\nu-(-\ln|x|)^\nu)\,,\;x\in \Omega_\vp,
	$$
for $\nu\in (0,1/2)$, and $\phi_\vp$ is extended by $0$ in $D_\vp$ and by $1$ outside of $D_1$. It is clear
that $\phi_\vp(x)=1$ where
	$$
	(-\ln|x|)^\nu\le (-\ln\vp)^\nu-1\Leftrightarrow |x|\ge \exp\left(-((-\ln\vp)^\nu-1)^{1/\nu}\right)=:r(\vp).
	$$
Notice that $r(\vp)\to 0$ as $\vp\to 0$.
Thus, using polar coordinates, we have
	\begin{align}
	\int_\Omega {}&|\nabla\phi_\vp|^2=2\pi\int_{\vp}^{r(\vp)}\left(\nu(-\ln \rho)^{\nu-1}\times
\frac{-1}{\rho}\right)^2 \rho \d \rho\nonumber\\
={}&2\pi\int_{\vp}^{r(\vp)}\nu^2 (-\ln \rho)^{2\nu-2} \frac{\d \rho}{\rho}
=-\frac{2\pi\nu^2}{2\nu-1}(-\ln \rho)^{2\nu-1}\Big|^{r(\vp)}_\vp\to 0
\label{grad.phieps}
	\end{align}
when $\vp\to 0$. Moreover $\phi_\vp\to 1$ a.e.\ on $\Omega$ while remaining
bounded by $1$. Assume that $v \in H^1_\per(\Omega) \cap L^\infty(\Omega)$.
Then by dominated convergence $\phi_\vp v\to v$ in $L^2(\Omega)$ as $\vp\to 0$.
Moreover, $\nabla(\phi_\vp v)=(\nabla\phi_\vp)v+\phi_\vp\nabla v$
so that, using $v\in L^\infty(\Omega)$ and \eqref{grad.phieps} for the first term and
the dominated convergence for the second term, $\nabla(\phi_\vp v)\to \nabla v$
in $L^2(\Omega)$. Hence,
\[
\phi_\vp v\to v\mbox{ in $H^1_\per(\Omega)$ as $\vp\to 0$}.
\]
  Finally, given $ v \in H^1_\per(\Omega) $, note  that $v_n=\max(-n,\min(v,n))  \in H^1_\per(\Omega) \cap L^\infty(\Omega) $  converges to $ v$ in $H^1(\Omega)$ as $n\to\infty$. This  allows us to deduce the existence of the required sequence using a diagonal argument.
\end{proof}

We remark that we have shown that $\cup_{\vp>0} V_{0,\vp}$ is dense in $H^1_\per(\Omega)$ in the
strong topology.
We are now in a position to prove our first convergence result.

\begin{proof}[Proof (Theorem \ref{mainlaplace})]

For fixed $ r > 0 $, the existence and uniqueness of $ u_r $ follow from the Lax--Milgram Lemma and Lemma \ref{badpoinc}.

We consider the cases when $\int_\Omega f=0$ and $\int_\Omega f\neq0$ separately.

\textbf{a) Assume that $\int_\Omega f=0$.} We first obtain an estimate for the solution $ u_r $. By taking $ v = u _r $ in \eqref{Lapr} and using the Poincar\'e inequality \eqref{perpoinc} one has
  \begin{align*}
  \|\nabla u_r\|^2_{L^2}=\int_{\Omega}|\nabla u_r|^2&=\int_\Omega fu_r\\
  &=\int_\Omega f\left(u_r-\fint u_r\right)\\
  &\le\|f\|_{L^2}\left\|u_r-\fint u_r\right\|_{L^2}\le C\|f\|_{L^2}\|\nabla u_r\|_{L^2},
  \end{align*}
  from which it follows that
  \begin{equation}\label{uniform}
  \|\nabla u_r\|_{L^2}\le C\|f\|_{L^2},
  \end{equation}
    with a constant $C>0$   independent on $ r$.

Next, define
$$
\tilde u_r=u_r-\fint u_r.
$$
Then from the bound \eqref{uniform} and the Poincar\'e inequality \eqref{perpoinc}, $\|\tilde u_r\|_{H^1(\Omega_r)}$ is uniformly bounded.

It follows that, up to the extraction of a subsequence,  $\nabla u_r=\nabla\tilde u_r\cgw \nabla u_0$ and $\tilde u_r\to u_0$ in $L^2(\Omega)$. Note that
\begin{equation}\label{u00av}
\int_\Omega u_0=\lim_{r\to0}\int_\Omega\tilde u_r=\lim_{r\to0}\int_\Omega \left(u_r-\fint u_r\right)=0.
\end{equation}

Now, we pass to the limit in the weak formulation \eqref{Lapr}.  Fix $r_0>0$ and observe that for $r<r_0$ one has $V_{0,r_0}\subset V_{0,r}$. Thus,
  $$
  \int_\Omega\nabla u_r\cdot\nabla v=\int_\Omega fv\qquad\mbox{for all}\ v\in V_{0,r_0}.
  $$
The weak convergence of $\nabla u_r$ to $\nabla u_0$ in $L^2(\Omega)$ allows us to pass to the limit and  obtain
 \begin{equation}\label{limit1.r0}
  \int_\Omega\nabla u_0\cdot\nabla v=\int_\Omega fv\qquad\mbox{for all}\ v\in V_{0,r_0}\,,\mbox{ for all $r_0>0$}.
  \end{equation}
From Lemma \ref{approx}, given $ v \in H^1_\per(\Omega) $ there exists a sequence of test functions $v_\vp \in V_{0,\vp}$  such that
$ v_\vp \to v $ in $ H^1(\Omega).$ Thus, by \eqref{limit1.r0},
  $$
  \int_\Omega\nabla u_0\cdot\nabla v_\vp=\int_\Omega fv_\vp .
  $$
 Passing to the limit as $ \vp \to 0$,  it follows that
  $$
  \int_\Omega\nabla u_0\cdot\nabla v=\int_\Omega fv\qquad\mbox{for all}\ v\in H^1_\per(\Omega),
  $$
  as claimed.

  Since the limiting problem has a unique solution when one imposes the zero average condition, it follows that all convergent subsequences must have the same limit. As a consequence, the original sequence converges without the need to extract a subsequence.

  It remains to show that in fact $\nabla u_r\to\nabla u_0$ in $L^2(\Omega)$ as $ r \to 0$. To this end we show that $\|\nabla u_r\|^2_{L^2}\to\|\nabla u_0\|^2_{L^2}$. Since $u_r-\fint u_r\to u_0$ in $L^2(\Omega)$,
    $$
    \int_{\Omega_r}|\nabla u_r|^2=\int_{\Omega_r} fu_r=\int_\Omega fu_r= \int_\Omega f\left(u_r-\fint u_r\right)\to\int_\Omega fu_0.
    $$
 However, from \eqref{Lap0} we have
    $$
    \int_\Omega|\nabla u_0|^2=\int_\Omega fu_0,
    $$
    which implies that
    $$
    \int_\Omega|\nabla u_r|^2\to\int_\Omega|\nabla u_0|^2.
    $$
    Coupled with weak convergence this norm convergence implies strong convergence of $\nabla u_r$ to $\nabla u_0$ in $L^2(\Omega)$.

\textbf{b) Assume that $\int_\Omega f \neq 0$.}  We note here that if $\int_\Omega f\neq0$ and one assumes a uniform bound on $\|\nabla u_r\|_{L^2}$, then one can follow the above argument (apart from obtaining the zero average condition \eqref{u00av}) to show that there is a solution of the limiting problem. But as remarked after the statement of Theorem \ref{mainlaplace}, there can be no such solution. It follows that in this case $\|\nabla u_r\|_{L^2}$ cannot be uniformly bounded as $r\to0$.\end{proof}

We note that in fact $\|\nabla u_r\|_{L^2}$ increases as $r$ decreases. Indeed, note that if $r'<r$ then $V_{0,r}\subset V_{0,r'}$. So we can take $v=u_r$ in both formulations
  $$
  \int_{\Omega_r}\nabla u_r\cdot\nabla v=\int_{\Omega_r}f  v\quad\mbox{and}\quad\int_{\Omega_{r'}}\nabla u_{r'}\cdot\nabla v=\int_{\Omega_{r'}}f v
  $$
  to obtain
  $$
  \int_{\Omega_r}|\nabla u_r|^2=\int_{\Omega_r}f  u_r\quad\mbox{and}\quad\int_{\Omega_{r'}}\nabla u_{r'}\cdot\nabla u_r=\int_{\Omega_{r'}}f  u_r=\int_{\Omega_r}f u_r.
  $$
  Thus
  $$
  \int_{\Omega_r}|\nabla u_r|^2=\int_{\Omega_{r'}}\nabla u_{r'}\cdot\nabla u_r
  $$
  whence
  $$
  \|\nabla u_r\|_{L^2(\Omega_r)}^2\le\|\nabla u_{r'}\|_{L^2(\Omega_{r'})}\|\nabla u_r\|_{L^2(\Omega_r)},
  $$
  i.e.
  $$
  \|\nabla u_r\|_{L^2(\Omega_r)}\le\|\nabla u_{r'}\|_{L^2(\Omega_{r'})}.
  $$

  \subsection{Failure of `uniform elliptic regularity'}

The Poisson equation enjoys elliptic estimates on the second derivatives.
Here we describe an example that shows that, for a punctured domain (with a slightly
different geometry to that in \eqref{laplacebis}), such estimates
may not be uniform with respect to the size of the hole. We consider the annulus (`punctured disc')
$$
\Omega_\vp=B(0,2)\setminus B(0,\vp),
$$
    with Dirichlet conditions on the inner and outer boundary. We solve the Poisson equation in plane polar co-ordinates for radially symmetric solutions, using $'$ for $\d/\d r$:
$$
\frac{1}{r}(ru')'=f(r)\qquad u(\vp)=0,\qquad u(2)=0.
$$
We take $f=1-(3r/4)$ so that $\int_\Omega f\,\d x=\int_0^{2\pi}\int_0^2 rf(r)\,\d r\,\d\theta=0$.

Then
$$
(ru')'=r-\frac{3r^2}{4}\qquad\Rightarrow\qquad ru'(r)=\frac{r^2}{2}-\frac{r^3}{4}+C
$$
and so
$$
u'(r)=\frac{r}{2}-\frac{r^2}{4}+\frac{C}{r}.
$$
Integrating again we obtain
$$
u(r)=\frac{r^2}{4}-\frac{r^3}{12}-\frac{\vp^2}{4}+\frac{\vp^3}{12}+C\log(r/\vp),
$$
and the boundary condition at $r=2$ implies that
$$
C=\frac{1}{\log(2/\vp)}\left[-\frac{1}{3}+\frac{\vp^2}{4}-\frac{\vp^3}{12}\right].
$$

Rewrite the governing equation as
$$
u''+\frac{1}{r}u'=f.
$$
Then $\|u''\|_{L^2}$ is bounded by $\|f\|_{L^2}+\|r^{-1}u'\|_{L^2}$. So consider
\begin{align*}
\frac{u'(r)}{r}&=\frac{1}{2}-\frac{r}{4}-\frac{C}{r^2}.
\end{align*}
As the first two terms are in $L^2$, we need only consider the final term. Noting that
$$
\|r^{-1}u'\|_{L^2}^2=2\pi \int_\vp^2 r(r^{-1}u')^2\sim 2\pi C^2\int_\vp^2 \frac{1}{r^3}\sim C^2\vp^{-2},
$$
so $\|u\|_{\dot H^2}\sim \vp^{-1}(-\log\vp)^{-1}$ with log corrections.

One can find a similar example in the three-dimensional case, namely $f(r)=1-5r^2/3$ on the spherical shell between $r=\vp$ and $r=1$.

 The lack of such a bound unfortunately appears to invalidate the arguments treating a moving disc in \cite{DR} and a moving sphere in \cite{Silvestre}.

\section{The Stokes equations}

In this section we extend the results of the previous section to the Stokes problem
$$
-\Delta \uu_r + \nabla p_r=\ff\ \;\text{in } \Omega_r,\qquad \uu_r|_{\partial D_r}=0,\qquad {\rm div}\,\uu_r=0.
$$

First we introduce the required spaces of vector fields. Given any space of scalar functions $ X $ we write $ \X$ for the two-component space $X \times X.$
Define for $r\geq 0$
$$
\H^1_{\per}(\Omega_r)=\text{ the closure of }\C_{\per}^1(\overline \Omega_r) \text{ in } \H^1 (\Omega_r),
$$
$$
\H^1_{\per,\sigma}(\Omega_r)=\{\vv\in \H^1_{\per}(\Omega_r):\ \operatorname{div}\vv=0\ \mbox{in }\Omega_r\},
$$
$$
\V_{0,r}=\{\vv\in \H^1_{\per}(\Omega_r):\ \vv=0\ \mbox{on}\ \partial D_r\},
$$
and
$$
\V_{0,r,\sigma}=\{\vv\in \H^1_{\per,\sigma}(\Omega_r):\ \vv=0\ \mbox{on}\ \partial D_r\}.
$$
%and
%$$
%\H_{r, \sigma}= \{ \vv_r \in \L^2_\per(\Omega_r): \ \operatorname{div}\vv=0\ \mbox{in }\Omega_r \}.
%$$
We observe that any function belonging to $\V_{0,r}$ or
 $\V_{0,r,\sigma} $
%or $\H_{r,\sigma}$
can be extended by zero inside of $ D_r $ to give a function in
$\H^1_{\per}(\Omega)$ or
$\H^1_{\per,\sigma}(\Omega)$, % or $\L^2_{\per}(\Omega),$
respectively.

We will determine the asymptotic behaviour of weak solutions to the following Stokes problem when $r \to 0$ :
$$
-\Delta \uu_r + \nabla p_r=\ff\ \;\text{in } \Omega_r,\qquad \uu_r \in \V_{0,r,\sigma}.
$$

Our second convergence result is as follows. We use a colon in the left-hand side of (\ref{stok}) to denote summation in both indices,
$$
\nabla\uu:\nabla\vv=\sum_{i,j=1}^2(\partial_iu_j)(\partial_iv_j).
$$

\begin{thm}\label{mainstokes}
Let $ \ff \in \L^2(\Omega).$  For every $r>0$ there exists a unique solution $\uu_r\in \V_{0,r,\sigma}$ of the problem
  \begin{equation}\label{stok}
  \int_{\Omega_r}\nabla \uu_r:\nabla \vv=\int_{\Omega_r}\ff\cdot\vv\qquad\mbox{for all}\ \vv\in \V_{0,r,\sigma}.
  \end{equation}
    Moreover
\begin{itemize}
  \item[a)]if $\int_\Omega \ff=0$ then as $r\to0$
  $$
  \uu_r-\frac{1}{|\Omega|}\int_\Omega \uu_r\to \uu_0\qquad\mbox{and}\qquad \nabla \uu_r\to\nabla \uu_0,
$$
where the limits are taken in $\L^2(\Omega)$ and $\uu_0 \in  \H^1_{{\rm per},\sigma}(\Omega)$ is the unique solution of the problem
  \begin{equation}\label{stok0}
 \int_\Omega\nabla \uu_0:\nabla \vv=\int_\Omega \ff\cdot\vv\qquad\mbox{for all}\ \vv\in \H^1_{{\rm per},\sigma}(\Omega)
  \end{equation}
  that satisfies $\int_\Omega \uu_0=0$;

  \item[b)]
  if $\int_\Omega \ff\neq0$ then $\|\nabla \uu_r\|_{\L^2}$ is unbounded as $r\to0$.  %More precisely, $\|\nabla \uu_r\|_{\L^2}$ increases as $r$ decreases.

  \end{itemize}
\end{thm}

Note that if we set $\vv=(1,0)$ and $\vv=(0,1)$ as test functions in \eqref{stok0}, then one can see immediately that for
$$\int_\Omega \ff \ne 0$$  a solution cannot exist.

  The only difference from the Poisson problem is that we now have to approximate functions in $ \H^1_\per(\Omega)$ by functions in $ \V_{0,r,\sigma}$, i.e.\ we must incorporate the divergence-free condition. If we have such approximating functions then we can use the same argument as before to show convergence of solutions to those of the limiting problem. Indeed,
the Poincar\'e inequalities work the same way as before and if $\int_\Omega \ff=0$ then
  $$\|\nabla \uu_r\|_{\L^2}\le C\|\ff\|_{\L^2}, \;\forall r>0,$$
  where $C$ is a constant independent of $r$.

  To deal with the divergence-free issue,
 we consider the following divergence problem for $g \in L^2(\Omega)$, and $\int_{\Omega}g=0$:
\begin{equation}\begin{cases} \label{divprob}
\operatorname{div}\hh=g \quad \text{in } \Omega,\\
\hh\in \H_0^{1}(\Omega).
\end{cases}
\end{equation}
When $ \Omega $ is star-like with respect to every point
of $ D_R(x_0) $ with $ \overline{D}_R(x_0) \subset \Omega$, the existence of a solution $\ff$ of this problem is proved in \cite[Lem. III.3.1]{Galdi} together with the inequality
$$\|\hh\|_{\H_0^{1}(\Omega)} \le C\|g\|_{L^2(\Omega)},$$
where the constant $ C $ depends on $\ R $ and the diameter of $ \Omega$.
Note that the divergence problem does not have a unique solution, since by adding any divergence-free function that vanishes on the boundary to the function $\hh$ one would get another solution. Nevertheless, for more general bounded domains, for instance, those satisfying the cone condition, the following result is true (cf. \cite[Thm III.3.1, Rmk. III.3.1]{Galdi}).

\begin{thm}\label{prop}
Let $\Omega$ be a bounded domain in $\R^2$ such that $\Omega=\cup_{j=1}^nU_j$, where each $U_j$ is star-shaped with respect to some open ball $B_j$ with $\overline{B_j}\subset U_j$. Then, given $g\in L^2(\Omega)$ with $\int_\Omega g=0$, there exists at least one solution $\hh$ to
\eqref{divprob} satisfying
$$
\|\hh\|_{\H^{1}_0(\Omega)}\le C^*C\|g\|_{L^2(\Omega)},
$$
where $C$ depends on $n$, the diameter of $\Omega$ and the smallest radius of the balls $B_j$. The constant $C^*$ is the maximum of
$$
C_1=1+\left(\frac{|U_1|}{|F_1|}\right)^{1/2}
$$
and
$$
C_k=\left(1+\left(\frac{|U_k|}{|F_k|}\right)^{1/2}\right)\prod_{i=1}^{k-1}\left(1+
\left(\frac{|D_i\setminus U_i|}{|F_i|}\right)^{1/2}\right),\quad k\ge 2,
$$
where $D_i=\cup_{s={i+1}}^nU_s$ and $F_i=U_i\cap D_i$.
\end{thm}
%
%
% \begin{prop}{\label{prop}}
%Let $g\in L^p(\Omega) \cap L^k(\Omega)$, $1<p,k<\infty$, with $\int_\Omega g=0$. Then there exists a solution $\ff\in \W_0^{1,p}(\Omega)\cap \W_0^{1,k}(\Omega)$ to the divergence problem
%\eqref{divprob} satisfying
%$$\|\ff\|_{\W^{1,p}_0(\Omega)}\le C\|g\|_{L^p(\Omega)} \qquad\mbox{and}\qquad \|\ff\|_{\W^{1,k}_0(\Omega)}\le C\|g\|_{L^k(\Omega)}$$
%\end{prop}

We are going to apply this theorem to the domain $ \Omega_\varepsilon$. In this case, it is not difficult to see that the constant in the inequalities can be bounded independently of $ \varepsilon$, as follows. For some $\varepsilon>0$ consider the domain $\Omega_\varepsilon$. $U_0$ denotes the part enclosed by the dashed lines in the picture, which is a part of the covering. When we perform  rotations of $\frac{\pi}{2},~ \pi, ~\frac{3\pi}{2}$ of  $U_0$ we obtain a covering of $\Omega_\varepsilon$ by $U_0$, $U_1$, $U_2$, $U_3$. As $\varepsilon$ decreases the triangle $S_0$ increases and we can put a fixed ball in $S_0$ for all smaller $\varepsilon$, such that $U_0$ is star-like with respect to this ball (we can do the same in each $U_i$). Moreover, we can easily see that the real numbers $|F_i|$ can be bounded from below. Therefore, we see that the constants in Theorem \ref{prop} can be bounded independently of $\varepsilon$, as claimed.

\begin{figure}[ht!]
\centering
 \includegraphics[width=0.5\textwidth]{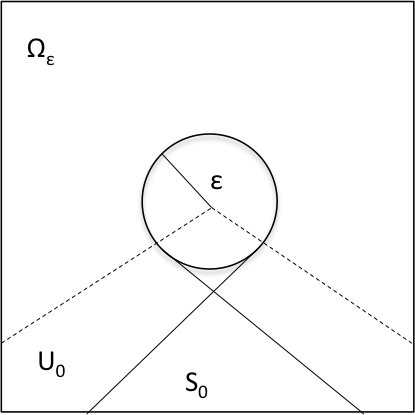}
 \caption{The constant in Theorem 3.2 can be taken to be bounded independently of $\varepsilon$.}
\end{figure}

We now prove the required lemma on the approximation of functions in $\H^1_{\per,\sigma}(\Omega)$ by functions in $\V_{0,\vp,\sigma}$.

\begin{lem}\label{approx2}
  If $\vv\in \H^1_{\per,\sigma}(\Omega)$ then there exists a sequence $\vv_\vp\in \V_{0,\vp,\sigma}$ such that
  $$
  \vv_\vp\to \vv\quad\mbox{in}\quad \H^1(\Omega)\qquad\mbox{as}\qquad\vp\to0.
  $$
\end{lem}

\begin{proof}
Let $\phi_\vp$ be the function introduced in Lemma \ref{approx}.
We first assume that $\vv\in \H^1_{\per,\sigma}(\Omega)\cap \L^\infty(\Omega)$. Then for $\vp$ small $\phi_\vp\vv\in \V_{0,\vp}$.
Since $\operatorname{div}(\vv)=0$ it follows
$$\operatorname{div}(\phi_\vp\vv)=\nabla \phi_\vp \cdot \vv.$$
Moreover,
\begin{flalign*}
\int_{\Omega_\vp}  \nabla \phi_\vp \cdot \vv= \int_{\Omega_\vp} \operatorname{div}(\vv\phi_\vp)=0.
\end{flalign*}
Noting that also that $\nabla \phi_\vp \cdot \vv $ belongs to $L^2(\Omega)$, it follows that it satisfies the conditions required by Theorem \ref{prop}, and so the divergence problem

$$\begin{cases}
\operatorname{div}\hh_\vp=-\nabla \phi_\vp \cdot \vv \quad \text{in } \Omega_\vp,\\
\hh_\vp\in \H_0^{1}(\Omega_\vp),
\end{cases}$$
has a solution $\hh_\vp$ satisfying
$$\|\hh_\vp \|_{\H_0^{1}(\Omega_\vp)} \le C \| \nabla \phi_\vp \cdot \vv \|_{L^2(\Omega_\vp)},$$
where $C$ depends only on $p$ and $\Omega$.
Setting $\vv_\vp=\hh_\vp+\phi_\vp \vv$ it is clear that $\vv_\vp\in \V_{0,\vp,\sigma}$
and, by Lemma \ref{approx}, that
\[
\vv_\vp\to \vv\mbox{ in $\H^1(\Omega)$ as $\vp\to0$.}
\]
It remains only to prove that a function in   $ \H^1_{\per,\sigma}(\Omega)$ can be approximated by functions in $ \H^1_{\per,\sigma}(\Omega)\cap \L^\infty(\Omega)$
which will allow us to conclude via a diagonal argument.

Let $\vv\in \H_{\per,\sigma}(\Omega)$ supposed to be extended by periodicity to $\R^2$.
Let $\varrho_n$ be a standard mollifier, i.e.\ $\varrho_n(x)=n^2\varrho(nx)$
where $\varrho$ is a $C^\infty$ function with support in the unit ball and such that
$$
\varrho\ge 0\,,\quad \int_{\R^2}\varrho =1.
$$
Then set
$$
\vv_n(x)=\varrho_n*\vv(x)=\int_{\R^2} \varrho_n(y)\vv(x-y)\d y.
$$
It is clear that $\vv_n$ is periodic in $x$ -- with the same period as $\vv$ -- divergence
free, smooth (and thus in $\L^\infty(\Omega)$) and, as $n\to\infty$,
$$
\vv_n,\, \nabla\vv_n \to \vv,\,\nabla \vv\mbox{ in $\L^2(\Omega)$ and  $\L^2(\Omega)^2$, respectively}.
$$
This completes the proof.
\end{proof}

To prove Theorem \ref{mainstokes} we essentially recapitulate the proof of Theorem \ref{mainlaplace} in this new setting.

\begin{proof}[ Proof (Theorem \ref{mainstokes})] Define
$$
\tilde \uu_r=\uu_r-\fint \uu_r.
$$
Then from the Poincar\'e inequality, $\|\tilde \uu_r\|_{\H^1(\Omega_r)}$ is uniformly bounded. Therefore for a subsequence $\nabla \uu_r=\nabla\tilde \uu_r\cgw \nabla \uu_0$ in $\H^1(\Omega)$ and $\tilde \uu_r\to \uu_0$ in $\L^2(\Omega)$, where $\uu_0$ satisfies $\int_\Omega \uu_0=0$.

  For a fixed $r_0$, $\forall r<r_0$ one has $\V_{0,\sigma, r_0}\subset \V_{0,\sigma,r}$. Thus
  $$
  \int_\Omega\nabla \uu_r:\nabla \vv=\int_\Omega \ff\cdot\vv\qquad\mbox{for all}\ \vv\in \V_{0,\sigma,r_0}.
  $$
  Passing to the limit in $r$ we obtain
  \begin{equation}\label{limit1}
  \int_\Omega\nabla \uu_0:\nabla \vv=\int_\Omega \ff\cdot\vv\qquad\mbox{for all}\ \vv\in \V_{0,\sigma,r_0}.
  \end{equation}

 Let $\vv\in \H^1_{\per,\sigma}(\Omega)$ and let $\vv_\vp$ be the approximating sequence from Lemma \ref{approx2}. Then for $\vp \leq r_0$ we have
  $$
  \int_\Omega\nabla \uu_0:\nabla \vv_\vp
 =\int_\Omega \ff\cdot\vv_\vp
  $$
  and passing to the limit in $\vp$ we obtain
  $$
  \int_\Omega\nabla \uu_0:\nabla \vv=\int_\Omega \ff\cdot\vv\qquad\mbox{for all}\ \vv\in \H^1_{\per,\sigma}(\Omega)
  $$
as required. (This is \eqref{stok0}.)

  Since the limiting problem has a unique solution when one imposes the zero average condition, it follows that all convergent subsequences must have the same limit. As a consequence, the whole original sequence converges toward $\uu_0$.

  To see that $\nabla \uu_r\to\nabla \uu_0$ in $\L^2(\Omega)$ we show that $\|\nabla \uu_r\|^2_{\L^2(\Omega)}\to\|\nabla \uu_0\|^2_{\L^2(\Omega)}$. Since $\uu_r-\fint \uu_r\to \uu_0$ in $\L^2(\Omega)$,
    $$
    \int_{\Omega_r}|\nabla \uu_r|^2=\int_{\Omega_r} \ff\cdot\uu_r=\int_\Omega \ff\cdot\uu_r=\int_\Omega \ff\cdot\left(\uu_r-\fint \uu_r\right)\to\int_\Omega \ff\cdot\uu_0.
    $$
    But from (\ref{stok0}) we have
    $$
    \int_\Omega|\nabla \uu_0|^2=\int_\Omega \ff\cdot\uu_0,
    $$
    which implies that
    $$
    \int_\Omega|\nabla \uu_r|^2\to\int_\Omega|\nabla \uu_0|^2.
    $$
    Coupled with weak convergence this implies strong convergence of $\nabla \uu_r$ to $\nabla \uu_0$ in $\L^2(\Omega)$.\end{proof}

\section{The time-dependent Navier--Stokes equations}

In this section we tackle the vanishing obstacle problem for the Navier--Stokes equations. The corresponding problem in a two-dimensional exterior domain (i.e.\ $\R^2\setminus D_r$)
was analysed in \cite{ILL2006} with the initial condition for the velocity corresponding to a fixed initial vorticity (independent of $r$). Here, by considering a periodic domain and suitable initial data we provide a less technical proof by using arguments along the lines of the previous sections.

We consider weak solutions to the following Navier--Stokes problem
\begin{flalign}
\begin{cases} \label{NSr}
\partial_t\uu_r -\Delta \uu_r + (\uu_r \cdot \nabla) \uu_r + \nabla p_r=\ff\ \;\text{in } \Omega_r \times (0,\infty),\\
\operatorname{div} \uu_r  = 0 \;\text{in } \Omega_r \times (0,\infty),\\
\uu_r = 0 \;\text{in } \partial D_r \times (0,\infty), \\ \text{periodic}, \\
\uu_r(0) = \uu_r^0  \;\text{in } \Omega_r,
\end{cases}
\end{flalign}
and show that they converge to periodic solutions of the equations on $\Omega$. Note that in this section we do not require that $\int_\Omega\ff=0$.

%In the course of the proof we will require the following lemma.
%
%\begin{lem}\label{useful}
%Suppose that $ \uu_n $ is bounded in $ L^2(0,T;\L^4(\Omega)) $ and converges to $ \uu $ in  $L^2(0,T;\L^2_{{\rm loc}} (\Omega \backslash \{0\})) $. Then
%$ \uu_n $ converges to $\uu$ in $  L^2(0,T;\L^2(\Omega)) $.
%\end{lem}
%
%\begin{proof} Let $ C > 0 $ be such that $\| \uu_n \|_{L^2(0,T;\L^4(\Omega))}^2$, $\| \uu \|_{L^2(0,T;\L^4(\Omega))}^2 \leq C .$
%Given $ \vp > 0 $ we can take an open set $ K_\vp \subset \subset \Omega \backslash \{0\} $ such that $ |  \Omega \backslash K_\vp | \leq  \frac{\vp^2}{8C^2}.$
%Now, there exists  $ n_\vp $    such that
%\[
%\int_0^T\int_{K_\vp} | \uu_n - \uu |^2 \leq \frac{\vp}{2}\qquad \text{ for } n \geq n_\vp.
%\]
%Thus, for $ n \geq n_\vp$,
%\begin{align*}
%\int_0^T\int_{\Omega} | \uu_n - \uu |^2 & = \int_0^T \int_{K_\vp} | \uu_n - \uu |^2 + \int_0^T\int_{\Omega \backslash K_\vp} | \uu_n - \uu |^2 \\
%& \leq \frac{\vp}{2} + |  \Omega \backslash K_\vp |^{\frac{1}{2}} \int_0^T  \Bigl(\int_{\Omega} (|\uu_n|^4 + |\uu|^4)\Bigr)^{\frac{1}{2}}\leq  \vp.\qedhere
%\end{align*}
%\end{proof}
%
We introduce the spaces
$$
\H_{r,\sigma}=\text{ the closure of } \{ \vv \in \C_{\per}^1(\overline \Omega_r) :\ \vv=0\ \mbox{on}\ \partial D_r ,\ \operatorname{div}\vv=0\ \mbox{in }\Omega_r\}  \text{ in } \L^2 (\Omega_r)
$$
and
$$
\H_{\sigma}= \H_{0,\sigma} = \{ \vv \in \L_{\per}^2( \Omega) :\ \operatorname{div}\vv=0 \}.
$$

We can now prove our convergence result for time-dependent Navier--Stokes solutions.
%$$
%\H_{r, \sigma}= \{ \vv \in \L^2_\per(\Omega_r): \ \operatorname{div}\vv=0\ \mbox{in }\Omega_r, \ \vv \cdot n = 0 \  \mbox{on } \partial D_r \},
%$$
%where $n$  is the exterior normal.
%Let $P_r$  denote the  orthogonal projection of $ \L^2(\Omega_r)$ onto $\H_{r, \sigma}$.

\begin{thm}\label{mainNS}
Let $ T > 0$, $ \uu_r^0 \in \H_{r, \sigma} $ and $ f \in \L^2((0,T) \times \Omega)$.  For every $r>0$ there exists a unique weak solution $\uu_r$ of  problem \eqref{NSr}, i.e.\ a unique $\uu_r\in L^2(0,T;\V_{0,r,\sigma}) \cap L^\infty(0,T;\H_{r,\sigma})$ with $\partial_t\uu_r\in L^2(0,T;\V'_{0,r,\sigma})$, such that
  \begin{align}\label{weakNSr}
&\langle\partial_t \uu_r , \vv\rangle +  \int_{\Omega_r}\nabla \uu_r:\nabla \vv + \int_{\Omega_r} [(\uu_r\cdot\nabla) \uu_r]\cdot \vv=\int_{\Omega_r}\ff\cdot\vv\ \mbox{ for all } \vv\in \V_{0,r,\sigma}, \\
&\uu_r(0)= \uu_r^0.
  \end{align}
In addition, $\uu_r$ satisfies the energy inequality
\begin{equation} \label{energy}
\| \uu_r (t) \|_{\L^2(\Omega_r)}^2 +  \int_0^t \| \nabla \uu_r\|_{\L^2(\Omega_r)}^2 \leq C(T) ( \|\uu_r^0 \|_{\L^2(\Omega_r)}^2 + \int_0^t \|\ff\|_{\L^2(\Omega_r)}^2).
\end{equation}
Furthermore, if $\uu_r^0 \rightharpoonup  \uu^0 $ in $\L^2(\Omega) $  as $r\to0$, then
  $$
  \uu_r \to \uu \mbox{ strongly in }\L^2(0,T;\H_\sigma) \text{ and weakly in } L^2 (0,T;\H^1_{\per, \sigma}(\Omega)),
$$
where  $\uu$ is the unique weak solution of the Navier--Stokes problem
   \begin{align*}
&\langle\partial_t \uu, \vv\rangle +  \int_{\Omega}\nabla \uu:\nabla \vv + \int_{\Omega} [(\uu\cdot\nabla) \uu]\cdot \vv=\int_{\Omega}\ff\cdot\vv \mbox{ for all}\ \vv\in \H^1_{\per, \sigma}(\Omega), \\
& \uu(0)= \uu^0.
  \end{align*}
  \end{thm}

\begin{proof}
The proof of existence of weak solutions follows  by using the Galerkin method and, since we are in dimension two, the uniqueness is also standard.
The energy inequality, which follows formally from the differential inequality
 $$
\partial_t\|\uu_r\|_{\L^2(\Omega_r)}^2+2\|\nabla\uu_r\|_{\L^2(\Omega_r)}^2\le\|\ff\|_{\L^2(\Omega)}^2+\|\uu_r\|_{\L^2(\Omega_r)}^2
$$
using the Gronwall lemma, follows rigorously from the same limiting Galerkin procedure, with an energy inequality obtained for each approximation. (See Constantin \& Foias \cite{CF}, Galdi \cite{G}, or Robinson \cite{JCRbook}, for example.)

We split the proof of convergence into three steps.  Briefly, we will obtain estimates for the solution $ \uu_r $ independent of $ r $, show that $\uu_r$ converges to a limit in various senses, and show this  is sufficient to pass to the limit in the weak formulation of the problem.

\noindent\textbf{Step 1: Estimates.}
From the energy inequality \eqref{energy} we already know that
\begin{equation} \label{estenergy}
 \uu_r \text{ is bounded in } L^\infty(0,T;\H_\sigma) \cap L^2(0,T;\H^1_{\per,\sigma}(\Omega))
 \end{equation}
 uniformly for $r>0$.  Recall that $ \uu_r  $ has been extended by zero inside $D_r$.

We need some strong convergence in order to pass to the limit in the nonlinear term. To this end, we first estimate the time derivative of
$\uu_r$ from \eqref{weakNSr}. Observe that
$$ \int_{\Omega_r} [(\uu_r\cdot\nabla) \uu_r]\cdot \vv = - \int_{\Omega_r} [(\uu_r\cdot\nabla) \vv]\cdot\uu_r, \ \text{ for all } \vv\in \V_{0,r,\sigma}. $$
%Next,  for a fixed $r_0$, $\forall r \leq r_0$ one has $\V_{0, r_0,\sigma}\subset \V_{0,r,\sigma}$.
Thus, for any $\vv \in \V_{0, r,\sigma}$
  \begin{align} \label{tempest}
|\langle\partial_t \uu_r , \vv\rangle | & = \Bigl| -\int_{\Omega_r}\nabla \uu_r:\nabla \vv + \int_{\Omega_r} [(\uu_r\cdot\nabla) \vv]\cdot\uu_r +\int_{\Omega_r}\ff \cdot \vv \Bigr|  \nonumber \\
& \leq C ( \|\nabla \uu_r \|_{\L^2(\Omega)} + \|\uu_r\|_{\L^2(\Omega)}\| \uu_r \|_{\H^1(\Omega)} + \|\ff\|_{\L^2(\Omega)} ) \| \vv\|_{\H^1(\Omega)}  \nonumber \\
& \leq C ( \| \uu_r \|_{\H^1(\Omega)} + \|\ff\|_{\L^2(\Omega)} ) \| \vv\|_{\H^1(\Omega)}, \text{ a.e. } t ,
\end{align}
where we have used the interpolation inequality
\[ \| \uu \|_{\L^4(\Omega)} \leq C \| \uu \|_{\L^2(\Omega)}^{\frac{1}{2}}  \| \uu \|_{\H^1(\Omega)}^{\frac{1}{2}} \]
and that $ \uu_r $ is uniformly  bounded in $L^\infty(0,T;\H_\sigma)$.
%Then, by \eqref{estenergy}, we deduce that
%\begin{equation} \label{tempder}
%\partial_t \uu_r \text{ is bounded in }  L^2(0,T; \V'_{0,r_0,\sigma})
%\end{equation}
%uniformly for $r>0$.

Next, we claim that
\[ \| \uu_r (\cdot +h) - \uu_r(\cdot)\|_{L^2(0,T-h;\L^2(\Omega))}^2 \leq C h .\]
Indeed,
\begin{align*}
& \| \uu_r (\cdot+h) - \uu_r(\cdot)\|_{L^2(0,T-h;\L^2(\Omega))}^2 \\
& = \int_{0}^{T-h} \langle \uu_r (t+h) - \uu_r(t), \uu_r (t+h) - \uu_r(t) \rangle  dt \\
& = \int_{0}^{T-h} \langle \int_{t}^{t+h} \partial_t\uu_r (s)ds, \uu_r (t+h) - \uu_r(t) \rangle  dt \\
& = \int_{0}^{T-h}\int_{t}^{t+h} \langle  \partial_t\uu_r (s), \uu_r (t+h) - \uu_r(t) \rangle  dsdt.
\end{align*}
Note that we have used that $ \int_{\Omega_r} \ww \cdot \vv  = \langle \ww, \vv\rangle  $ for $ \ww \in \H_{r,\sigma} $ and $ \vv \in \V_{0, r,\sigma}$.
As $ \uu_r (t+h) - \uu_r(t) \in \V_{0, r,\sigma} $  a.e. $t$, we can use estimate \eqref{tempest}. Thus, by applying  Young inequality and Fubini Theorem, we arrive at
\begin{align*}
& \| \uu_r (\cdot+h) - \uu_r(\cdot)\|_{L^2(0,T-h;\L^2(\Omega))}^2 \\
& \leq \int_{0}^{T-h} \int_{t}^{t+h}  ( \| \uu_r (s)\|_{\H^1(\Omega)} + \|\ff(s)\|_{\L^2(\Omega)} ) \| \uu_r(t+h) - \uu_r(t)\|_{\H^1(\Omega)} ds dt \\
&\leq  \int_{0}^{T-h} \int_{t}^{t+h}   \| \uu_r (s)\|_{\H^1(\Omega)}^2 + \|\ff(s)\|_{\L^2(\Omega)}^2 +  \| \uu_r(t+h) \|_{\H^1(\Omega)}^2 + \| \uu_r(t)\|_{\H^1(\Omega)}^2 ds dt \\
 &\leq  ( \|\ff\|_{L^2(0,T;\L^2(\Omega))}^2   + 3\| \uu_r\|_{L^2(0,T;\H^1(\Omega))}^2 ) h  \\
 & \leq C h
\end{align*}
where $ C$ is independent of $ r$. The claim is proved.

\noindent\textbf{Step 2: Convergence of $ \uu_r$.} Since $ \uu_r $ is bounded in $ L^2(0,T;\H^1_{\per,\sigma}(\Omega))$,
\[\| \uu_r (\cdot+h) - \uu_r(\cdot)\|_{L^2(0,T-h;\L^2(\Omega))} \to 0 \text{ as } h \to 0 \text{  uniformly in } r,  \]
and $ \H^1_{\per,\sigma}(\Omega) \subset \subset \H_\sigma$, we can apply Theorem 3 from \cite[p. 80]{simon} and conclude that
\[ \uu_r \text{ is relatively compact in } L^2(0,T;\H_\sigma).\]

Hence,  up to a subsequence, it holds
\begin{align*}
& \uu_r  \rightharpoonup \uu \text{ in } L^2(0,T;\H^1_{\per,\sigma}(\Omega)) \text{ and }\\
 & \uu_r  \rightarrow  \uu \text{ in } L^2(0,T;\H_\sigma)
\end{align*}
By interpolation and the H\"older inequality,
\begin{align*}
\int_0^T \|\uu_r - \uu \|_{\L^4(\Omega)}^2 & \leq C \int_0^T \|\uu_r - \uu \|_{\L^2(\Omega)}  \|\uu_r - \uu \|_{\H^1(\Omega)}  \\
& \leq C \Bigl(\int_0^T \|\uu_r - \uu \|_{\L^2(\Omega)}^2 \Bigr)^{\frac{1}{2}} .
\end{align*}
Thus, we infer in addition that
\begin{equation} \label{convL4}
\uu_r \rightarrow \uu \text{ in } L^2(0,T; \L^4( \Omega)).
\end{equation}

%\begin{align*}
%\int_0^T \|P_{r_0}\uu_r &- P_{r_0}\uu \|_{\L^4(\Omega_{r_0})}^2\\
% & \leq C \int_0^T \|P_{r_0}\uu_r - P_{r_0}\uu \|_{\L^2(\Omega_{r_0})}  \|P_{r_0}\uu_r - P_{r_0}\uu \|_{\H^1(\Omega_{r_0})}  \\
%& \leq C \Bigl(\int_0^T \| P_{r_0}\uu_r - P_{r_0}\uu \|_{\L^2(\Omega_{r_0})}^2 \Bigr)^{\frac{1}{2}} .
%\end{align*}
%Thus, we infer in addition that
%\begin{equation} \label{convL4}
%P_{r_0}\uu_r \rightarrow P_{r_0}\uu \text{ in } L^2(0,T; \L^4( \Omega_{r_0})).
%\end{equation}

\noindent\textbf{Step 3: Passage to the limit in the weak formulation.}
By using that, for a fixed $r_0$, $\forall r<r_0$ one has $\V_{0,r_0,\sigma}\subset \V_{0,r,\sigma}$, multiplying \eqref{weakNSr} by $\xi \in C^\infty_0[0,T)$ and integrating in time,
we have
\begin{align*}
-\int_0^T \int_\Omega  \uu_r\cdot  \vv \xi'  + \int_0^T \int_{\Omega}\nabla \uu_r:\nabla \vv \xi &-\int_0^T\int_{\Omega} [(\uu_r\cdot\nabla) \vv]\cdot \uu_r \xi \\
&=\int_0^T\int_{\Omega}\ff\cdot \vv \xi
+
\int_{\Omega} \uu_r^0\cdot \vv \xi(0)
\end{align*}
for all $ \vv \in \V_{0,r_0,\sigma} $ and $ \xi \in C^\infty_0[0,T)$.

The weak convergences are sufficient to pass the the limit in the linear terms. To show the convergence of the nonlinear term, we re-write
\begin{align*}
\int_0^T\int_{\Omega} [(\uu_r\cdot\nabla) \vv]&\cdot \uu_r \xi -[(\uu\cdot\nabla) \vv]\cdot \uu \xi\\
 &= \int_0^T\int_{\Omega} [((\uu_r-\uu)\cdot\nabla) \vv]\cdot \uu_r \xi + [(\uu \cdot\nabla) \vv]\cdot (\uu_r -\uu) \xi.
\end{align*}
We prove that the first term on the right-hand side goes to zero; the convergence of the second term is proved similarly.  By using the H\"older inequality in space and then in time, we have
\begin{align*}
\Bigl|\int_0^T\int_{\Omega} [((\uu_r-\uu)&\cdot\nabla) \vv]\cdot \uu_r \xi \Bigr|\\
 & \leq \int_0^T \|\uu_r-\uu\|_{\L^4(\Omega)} \|\nabla \vv \|_{\L^2(\Omega)} \|\uu_r \|_{\L^4(\Omega)}\|\xi\|_{L^\infty(0,T)} \\
& \leq C \Bigl(\int_0^T \|\uu_r-\uu\|_{\L^4(\Omega)}^2 \Bigr)^{\frac{1}{2}}  \Bigl(\int_0^T\|\uu_r \|_{\H^1(\Omega)}^2\Bigr)^{\frac{1}{2}},
\end{align*}
where we have used the embedding $ \H^1(\Omega) \subset \L^4(\Omega).$
The convergence follows from convergence \eqref{convL4} and estimate \eqref{estenergy}.

%To show the convergence of nonlinear term, we re-write
%\begin{align*}
%& \int_0^T\int_{\Omega} [(\uu_r\cdot\nabla) \vv]\cdot \uu_r \xi \\
%& = \int_0^T\int_{\Omega} [(P_{r_0}\uu_r\cdot\nabla) \vv]\cdot P_{r_0}\uu_r \xi
% + \int_0^T\int_{\Omega} [(P_{r_0}\uu_r\cdot\nabla) \vv]\cdot \nabla q_r \xi
% \\
% & + \int_0^T\int_{\Omega} [(\nabla q_r\cdot\nabla) \vv]\cdot P_{r_0}\uu_r \xi
%+ \int_0^T\int_{\Omega} [(\nabla q_r\cdot\nabla) \vv]\cdot \nabla q_r \xi.
%\end{align*}
%The last term vanishes since
%\begin{align*}
%\int_0^T\int_{\Omega} [(\nabla q_r\cdot\nabla) \vv]\cdot \nabla q_r \xi & = \int_0^T\int_{\Omega} (\nabla q_r \otimes \nabla q_r) : \nabla \vv \xi \\
%& = - \int_0^T\int_{\Omega} \text{div}(\nabla q_r \otimes \nabla q_r) \cdot \vv  \xi\\
%& = - \int_0^T\int_{\Omega} \frac{1}{2}\nabla (|\nabla q_r |^2) \cdot \vv \xi\\
%& = \int_0^T\int_{\Omega} \frac{1}{2}|\nabla q_r |^2(\nabla\cdot\vv)\xi = 0 .
%\end{align*}
%The previous convergences allows to pass to the limit in the other terms, so we can conclude that, as $ r \rightarrow 0$,
%\[ \int_0^T\int_{\Omega} [(\uu_r\cdot\nabla) \vv]\cdot \uu_r \xi \rightarrow \int_0^T\int_{\Omega} [(\uu\cdot\nabla) \vv]\cdot \uu \xi .\]
%

  Passing to the limit in $r$ we obtain
  \begin{align*}
-\int_0^T \int_\Omega  \uu\cdot  \vv \xi'  + \int_0^T \int_{\Omega}\nabla \uu:\nabla \vv \xi& - \int_0^T\int_{\Omega} [(\uu\cdot\nabla)\vv]\cdot \uu \xi\\
 &=\int_0^T\int_{\Omega}\ff\cdot \vv \xi
+
\int_{\Omega} \uu^0\cdot \vv \xi(0)
\end{align*}
for all $ \vv \in \V_{0,r_0,\sigma} $ and $ \xi \in C^\infty_0[0,T)$.

Next, we argue as in the Stokes problem by using the approximation from Lemma \ref{approx2}. Given $ \vv \in \H^1_{\per,\sigma}(\Omega)  $ there exist $ \vv^\vp \in \V_{0,\vp, \sigma} $
such that $ \vv^\vp \rightharpoonup \vv $ in $ \H^1_{\per,\sigma}(\Omega).$ Thus, for $ \vp \leq r_0 $ one has
\begin{align*}
-\int_0^T \int_\Omega  \uu\cdot  \vv^\vp \xi'  + \int_0^T \int_{\Omega}\nabla \uu:\nabla \vv^\vp \xi &- \int_0^T\int_{\Omega} [(\uu\cdot\nabla)\vv^\vp]\cdot \uu \xi\\
& =\int_0^T\int_{\Omega}\ff\cdot \vv^\vp \xi +
\int_{\Omega} \uu^0\cdot \vv^\vp \xi(0).
\end{align*}

  Passing to the limit in $\vp$ we get
\begin{align}
-\int_0^T \int_\Omega  \uu\cdot  \vv \xi'  + \int_0^T \int_{\Omega}\nabla \uu:\nabla \vv \xi &- \int_0^T\int_{\Omega} [(\uu\cdot\nabla)\vv]\cdot \uu \xi\nonumber\\
& =\int_0^T\int_{\Omega}\ff\cdot \vv \xi
+
\int_{\Omega} \uu^0\cdot \vv \xi(0)\label{weakweak}
\end{align}
for all $ \vv \in \H^1_{\per,\sigma}(\Omega)  $ and $ \xi \in C^\infty_0[0,T)$.

%Finally, given $ \vv \in \H^1_{\per,\sigma}(\Omega) $ there exists $ \vv_n \in \H^1_{\per,\sigma}(\Omega) \cap \L^\infty(\Omega) $ such that $ \vv_n \rightarrow \vv $ in $ \H^1(\Omega).$
%In the same way as above, we conclude that
%\begin{multline}\label{weakweak}
% -\int_0^T \int_\Omega  \uu  \vv \xi'  + \int_0^T \int_{\Omega}\nabla \uu\cdot\nabla \vv \xi - \int_0^T\int_{\Omega} \uu\cdot\nabla \vv \uu \xi
% \\ =\int_0^T\int_{\Omega}\ff \vv \xi
%+
%\int_{\Omega} \uu^0 \vv \xi(0)
%\end{multline}
%for all $ \vv \in \H^1_{\per,\sigma}(\Omega)  $ and $ \xi \in C^\infty_0[0,T)$.
%

In particular, since $ \uu \in L^2(0,T;\H^1_{\per,\sigma}(\Omega)) \cap L^\infty(0,T;\H_\sigma)$ we can take $ \xi \in C^\infty_0(0,T) $ in (\ref{weakweak}) and deduce that $ \partial_t \uu \in L^2(0,T; (\H^1_{\per,\sigma}(\Omega))')$, whence  $\uu$ satisfies
\[\langle\partial_t \uu, \vv\rangle +  \int_{\Omega}\nabla \uu:\nabla \vv + \int_{\Omega} [(\uu\cdot\nabla) \uu]\cdot \vv=\int_{\Omega}\ff\cdot\vv  \]
for all $ \vv \in \H^1_{\per,\sigma}(\Omega)  $.

It remains only to prove that $ \uu(0) = \uu^0$.  To see this, multiply the previous equality by $ \xi \in C^\infty_0[0,T)$ and integrate in time, to obtain
\begin{align*}
-\int_0^T \int_\Omega  \uu\cdot  \vv \xi'  + \int_0^T \int_{\Omega}\nabla \uu:\nabla \vv \xi &- \int_0^T\int_{\Omega} [(\uu\cdot\nabla) \vv]\cdot \uu \xi\\
 &=\int_0^T\int_{\Omega}\ff\cdot \vv \xi
+
\int_{\Omega} \uu(0)\cdot \vv \xi(0)
\end{align*}
for all $ \vv \in \H^1_{\per,\sigma}(\Omega)  $ and $ \xi \in C^\infty_0[0,T)$. Comparing with \eqref{weakweak} we conclude that $ \uu(0) = \uu^0$.
Notice also that  $ \uu \in C([0,T]; \H_\sigma) $.

  Since the limiting problem has a unique solution, it follows that all convergent subsequences must have the same limit. As a consequence, the whole original sequence converges toward $\uu$.
\end{proof}

\section{Conclusions}

We have analysed three models in a simple but unusual geometry, the `punctured periodic domain', showing that the influence of the obstacle, a disc of radius $r$, evaporates in the limit as $r\to0$.

Some interesting open problems remain. While the lack of a bound on the average of the solution $u_r$ over $\Omega$ (in both the Poisson and Stokes problems) that is uniform in $r$ appears initially to be only a mathematical curiosity, such a bound is central to tackling the stationary Navier--Stokes problem in this geometry.

The fact that there is no `uniform elliptic regularity' for the Laplacian or Stokes operator in this geometry means that the important `vanishing tracer' problem (cf.\ \cite{DR,Silvestre}) also remains open.
Very recently, Lacave \& Takahasi \cite{Lacave17} obtained a partial result in the two-dimensional case assuming that
the density of the solid is independent of $r$.  They employed some optimal $L^p-L^q$ decay estimates of the semigroup associated to the fluid-rigid body system.
We plan to return to this in a future paper.

\subsection*{Acknowledgments}

This article was written during a part time employment of MC at the S. M. Nikolskii Mathematical Institute of RUDN University, 6 Miklukho-Maklay St, Moscow, 117198.
The publication was supported by the Ministry of Education and Science of the Russian Federation (Agreement number 02.a03.21.0008). MC also thanks Monash University for
an invitation during which a part of this article was completed.
GP thanks the Mathematics Institute of the University of Warwick for
their kind hospitality during her visit there. GP was partially supported by FAPESP, grant 2013/00048-3 and CNPq, grant 306646/2015-3, Brazil. JCR was partially supported by an EPSRC Leadership Fellowship EP/G007470/1, and would like to thank MC and WM for their hospitality in Zurich.

%
%
%\begin{thebibliography}{1}
%\bibitem{test} A. B. C. Test, \textit{On a Test.} J. of Testing
%\textbf{88} (2000), 100--120.
%\bibitem{latex} G. Gr\"atzer, \textit{Math into \LaTeX.} 3rd Edition,
%Birkh\"auser, 2000.
%\end{thebibliography}

% ------------------------------------------------------------------------

\begin{thebibliography}{99}

\bibitem{Ariel} P. D. Ariel. On computation of the three-dimensional flow past a stretching sheet, Appl.
Math. Computat. 188 (2007), 1244-1250.

\bibitem{CF} P. Constantin and C. Foias, The Navier--Stokes equations. University of Chicago Press, 1988.

\bibitem{DR} M. Dashti and J. C. Robinson, The motion of a fluid-rigid disc system at the zero limit
of the rigid disc radius, Arch. Ration. Mech. Anal. 200 (2011), no. 1, 285-312.

\bibitem{Elde} J. D. Eldredge, Numerical simulation of the fluid dynamics of 2D rigid body motion with the vortex particle method,
J. Comput. Phys. 221 (2007), 626-648.

\bibitem{G} G.P. Galdi, An introduction to the Navier--Stokes initial-boundary value problem.  Fundamental directions in mathematical fluid mechanics,  1--70, Adv. Math. Fluid Mech., Birkh\"{a}user, Basel, 2000.


\bibitem{Galdi} G.P Galdi, An Introduction to the Mathematical Theory of the Navier--Stokes Equations, second edition. Springer, New York,  2011.

\bibitem{HKM} J. Heinonen, T. Kilpel\"ainen and O. Martio,
Nonlinear potential theory of degenerate elliptic equations.
Dover Publications, Inc., Mineola, NY, 2006. xii+404 pp. ISBN: 0-486-45050-3

\bibitem{Homen} D. Homentcovschi. Three-dimensional Oseen flow past a flat plate, Q. Appl. Math. 40
(1982), 137-149.

\bibitem{Hudson} J. D. Hudson and S. C. R. Dennis, The flow of a viscous incompressible fluid past a normal
flat plate at low and intermediate Reynolds numbers: the wake. J. Fluid Mech. 160 (1985),
369-383.

\bibitem{IK2009} D. Iftimie,  J. Kelliher. Remarks on the vanishing obstacle limit for a
3D viscous incompressible fluid, Proc. Amer. Math. Soc. 137 (2009), 685-694.

\bibitem{ILL2003} D. Iftimie, M. C. Lopes Filho and H. J. Nussenzveig Lopes, Two-dimensional incompressible
ideal flow around a small obstacle. Commun. PDEs 28 (2003), 349-379.

\bibitem{ILL2006} D. Iftimie, M. C. Lopes Filho and H. J. Nussenzveig Lopes, Two-dimensional incompressible
viscous flow around a small obstacle. Math. Annalen 336 (2006), 449-489.

\bibitem{ILL2009} D. Iftimie, M. C. Lopes Filho and H. J. Nussenzveig Lopes, Incompressible flow around a
small obstacle and the vanishing viscosity limit. Commun. Math. Phys. 287 (2009), 99-115.

\bibitem{Lacave09} C. Lacave. Two-dimensional incompressible ideal flow around a thin obstacle tending to a
curve, Annales Inst. H. Poincar´e Analyse Non Lin´eaire 26 (2009), 1121-1148.

\bibitem{Lacave2} C. Lacave, Two-dimensional incompressible viscous flow around a thin obstacle tending
to a curve, Proc. Roy. Soc. Edinburgh Sect. A 139 (2009), no. 6, 1237-1254.

\bibitem{Lacave3} C. Lacave,  3D viscous incompressible fluid around one thin obstacle, Proc. Amer. Math. Soc.  143  (2015), 2175–-2191.

\bibitem{Lacave17} C. Lacave and T. Takahashi, Small moving rigid body into a viscous incompressible fluid.
Arch. Ration. Mech. Anal. 223 (2017), no. 3, 1307-1335.

\bibitem{Ladyz} O.A. Ladyzhenskaya, The mathematical theory of viscous incompressible flow. Gordon and Breach Science Publishers, New York-London 1963.

\bibitem{Lopes} M. C. Lopes Filho, Vortex dynamics in a two-dimensional domain with holes and the small
obstacle limit. SIAM J. Math. Analysis 39 (2007), 422--436.

\bibitem{Lopes13} M. C. Lopes Filho, H. H. Nguyen and H. J. Nussenzveig Lopes, Incompressible and ideal 2D flow around a small obstacle with constant velocity at infinity. Quart. Appl. Math. 71 (2013) 679--687.


\bibitem{Lu16} Y. Lu, On uniform estimates for Laplace equation in balls with small holes,  Calc. Var. Partial Differential Equations 55 (2016), no. 5, Art. 110, 19 pp.

\bibitem{Lupreprint} Y. Lu, Uniform estimates for Stokes equations in domains with small holes and applications in homogenization problems, Preprint arXiv:1510.01678.

\bibitem{Pierce} D. Pierce, Photographic evidence of the formation and growth of vorticity behind plates accelerated from rest in still air, J. Fluid Mech., 11 (1961), 460--464.

    \bibitem{JCRbook} J.C. Robinson, Infinite-dimensional dynamical systems. Cambridge University Press, 2001.

\bibitem{Robinson} J.C. Robinson, A coupled particle-continuum model: well-posedness and the limit of
zero radius, Proc. R. Soc. Lond. Ser. A Math. Phys. Eng. Sci. 460 (2004), 1311-1334.

\bibitem{Silvestre} A. L. Silvestre and  T. Takahashi, The Motion of a Fluid-Rigid Ball System
at the Zero Limit of the Rigid Ball Radius, Arch. Rational Mech. Anal. 211 (2014) 991-1012.

\bibitem{Smith} S. H. Smith, A note on the boundary layer approach to the impulsively started flow past a
flat plate, J. Engng Math. 29 (1995), 195-202.

\bibitem{simon} J. Simon, Compact sets in the space $L^p(0,T;B)$,
Ann. Mat. Pura Appl. 146 (1987), 65-96.

\bibitem{Temam1} R. Temam, Navier--Stokes equations. Theory and numerical analysis. Studies in Mathematics and its Applications, Vol. 2. North-Holland Publishing Co., Amsterdam-New York-Oxford, 1977.

\bibitem{Temam2} R. Temam, Navier--Stokes equations and nonlinear functional analysis. CBMS-NSF Regional Conference Series in Applied Mathematics, 41. Society for Industrial and Applied Mathematics (SIAM), Philadelphia, PA, 1983.

\bibitem{Turfus} C. Turfus, Prandtl-Batchelor flow past a flat plate at normal incidence in a channel: inviscid
analysis, J. Fluid Mech. 249 (1993), 59-72.

\bibitem{Zannetti} L. Zannetti, Vortex equilibrium in flows past bluff bodies, J. Fluid Mech. 562 (2006), 151-171.

\end{thebibliography}
\end{document}